\documentclass[12pt,twoside]{amsart}
\usepackage{amssymb}

\usepackage[all]{xy}
\nonstopmode

\headsep=0.5cm  \numberwithin{equation}{section}
\font\tengothic=eufm10 scaled\magstep 1 \font\sevengothic=eufm7
scaled\magstep 1
\newfam\gothicfam
      \textfont\gothicfam=\tengothic
      \scriptfont\gothicfam=\sevengothic

\newtheorem{theorem}{Theorem}[section]

\newtheorem{proposition}[theorem]{Proposition}
\newtheorem{corollary}[theorem]{Corollary}

\theoremstyle{definition}
\newtheorem{definition}[theorem]{Definition}
\newtheorem{remark}[theorem]{Remark}
\newtheorem{problem}[theorem]{Problem}
\newtheorem{example}[theorem]{Example}

\newtheorem{notation}[theorem]{Notation}
\newtheorem{question}[theorem]{Question}
\newcommand{\codim}{\operatorname{codim}}

\newcommand{\rank}{\operatorname{rank}}

\newcommand{\cW}{{\mathcal W}}
\newcommand{\cK}{{\mathcal K}}

\newcommand{\cE}{{\mathcal E}}
\newcommand{\cU}{{\mathcal U}}

\newcommand{\cF}{{\mathcal F}}

\newcommand{\cO}{{\mathcal O}}
\newcommand{\cD}{{\mathcal D}}

\newcommand {\RR}{\mathbb{R}}
\newcommand {\ZZ}{\mathbb{Z}}

\newcommand {\PP}{\mathbb{P}}
\newcommand {\FF}{\mathbb{F}}

\def\mapright#1{\smash{ \mathop{\longrightarrow}
    \limits^{#1}}}

\begin{document}
\title[ Brill-Noether theory for moduli spaces]
 {Brill-Noether theory for moduli spaces of sheaves on algebraic varieties}

\author[L.\ Costa, R.M.\ Mir\'o-Roig]{L.\ Costa$^*$, R.M.\
Mir\'o-Roig$^{**}$}

\address{Facultat de Matem\`atiques,
Departament d'Algebra i Geometria, Gran Via de les Corts Catalanes
585, 08007 Barcelona, SPAIN } \email{costa@ub.edu, miro@ub.edu}

\date{\today}
\thanks{$^*$ Partially supported by MTM2007-61104.}
\thanks{$^{**}$ Partially supported by MTM2007-61104.}



\begin{abstract}
Let $X$ be a smooth projective variety of dimension $n$  and let
$H$ be an ample line bundle on $X$. Let $M_{X,H}(r;c_1, \cdots,
c_{s})$ be the moduli space of $H$-stable vector bundles $E$ on
$X$ of rank $r$ and  Chern classes $c_i(E)=c_i$ for $i=1, \cdots,
s:=min\{r,n\}$. We define the Brill-Noether filtration on
$M_{X,H}(r;c_1, \cdots, c_{s})$  as $W_{H}^{k}(r;c_1,\cdots, c_{s}
)= \{ E \in M_{X,H}(r;c_1, \cdots, c_{s}) | h^0(X,E) \geq k \}$
and we realize $W_{H}^{k}(r;c_1,\cdots, c_{s} )$ as the $k$th
determinantal variety of a morphism of vector bundles on
$M_{X,H}(r;c_1, \cdots, c_{s})$, provided $H^i(E)=0$ for $i \geq
2$ and  $E \in M_{X,H}(r;c_1, \cdots, c_{s})$. We also compute the
expected dimension of $W_{H}^{k}(r;c_1,\cdots, c_{s} )$. Very
surprisingly we will see that the Brill-Noether stratification
allow us to compare moduli spaces of vector bundles on Hirzebruch
surfaces stables with respect to different polarizations. We will
also study the Brill-Noether loci of the moduli space of instanton
bundles and we will see that they have the expected dimension.
\end{abstract}


\maketitle

\tableofcontents


 \section{Introduction} \label{intro}

Let $X$ be a smooth projective variety of dimension $n$ over an
algebraically closed field $K$  of characteristic 0 and let
$M_{X,H}(r;c_1, \cdots, c_{s})$ be the moduli space of rank $r$,
vector bundles $E$ on $X$ stable with respect to an ample line
bundle $H$ and with fixed Chern classes $c_i(E)=c_i$ for $i=1,
\cdots, s:=min\{r,n\}$. Moduli spaces of stable vector bundles
were constructed in the 1970's by Maruyama and since then they
have been extensively studied from the point of view of algebraic
geometry, of topology and of differential geometry; giving very
pleasant connections between these areas. Unfortunately, except in
the classical case of vector bundles on curves, relatively little
is known about their geometry in terms of the existence and
structure of their subvarieties.

In the case of line bundles on smooth projective curves $C$ of
genus $g$, where the moduli spaces $Pic^d(C)$ of degree $d$ line
bundles are all isomorphic to the Jacobian, Brill-Noether theory
has long provided a basic source of geometrical information. The
classical theory of Brill-Noether has a long history and it is
concerned with the subvarieties $W^k$ of $Pic^d(C)$ whose line
bundles have  at least $k+1$ independent global sections. Basic
questions concerning non-emptiness, connectedness, irreducibility,
dimension, singularities, cohomological classes etc ... have been
answered when the curve $C$ is generic on the moduli space ${
M}_g$ of curves of genus $g$. There are several natural ways to
generalize the classical theory of Brill-Noether. First, instead
of studying line bundles on curves, we can consider vector bundles
of any rank giving rise to the Brill-Noether loci in the moduli
space of stable rank $r$ vector bundles on curves studied by
Newstead, Teixidor and others. Indeed, during the last two
decades, a great amount of job has been made around the
Brill-Noether stratification of the moduli space of degree $d$,
stable, rank $r$ vector bundles on algebraic curves, giving rise
to nice and interesting descriptions of these subvarieties.
Nevertheless, it should be mentioned that in spite of big efforts,
many questions concerning their geometry still remain open.
Second, instead of studying line bundles on curves, we can
consider line bundles on varieties of arbitrary dimension and,
finally, we can go in both directions simultaneously. We can
consider a smooth projective variety $X$ of dimension $n$, an
ample line bundle $H$ on $X$, the moduli space $M_{X,H}(r;c_1,
\cdots, c_{s})$ of rank $r$, $H$-stable vector bundles $E$ on $X$
with fixed Chern classes $c_i(E)=c_i$; $1\le i \le min\{r,n\}$,
and we can study the subschemes in $M_{X,H}(r;c_1, \cdots, c_{s})$
defined by conditions $\{ \dim H^{j}(X,E)\ge n_{j} \}$. In
\cite{GoHi}, Gottsche and Hirschowitz studied the Brill-Noether
loci in the moduli space of stable vector bundles on $\PP^2$ and
in \cite{L}-\cite{L1}, Leyenson studied the Brill-Noether loci in
the moduli space of stable vector bundles on K3 surfaces. The goal
of this paper is to introduce a Brill-Noether theory for moduli
spaces of rank $r$, $H$-stable vector bundles on algebraic
varieties of arbitrary dimension, extending, in particular, all
the above results to higher dimensional varieties. Once we will
have constructed the Brill-Noether stratification in this much
more general context, we will address the main problems and we
will analyze them for several concrete moduli problems.

\vskip 2mm Next we outline the structure of this paper. In section
2, we will define the  Brill-Noether locus
$W_{H}^{k}(r;c_1,\cdots, c_{s} )$ in $M_{X,H}(r;c_1, \cdots,
c_{s})$  as the set of vector bundles in $M_{X,H}(r;c_1, \cdots,
c_{s})$ having at least $k$ independent sections and associated to
this locus we consider the generalized Brill-Noether number
$\rho^{k}_H(r;c_1,\cdots, c_{s} )$. We prove that
$W_{H}^{k}(r;c_1,\cdots, c_{s} )$ has a natural structure of
algebraic variety and that any of its non-empty components has
dimension $\ge \rho^{k}_H(r;c_1,\cdots, c_{s} )$. Therefore, it is
natural to ask whether the numerical condition
$\rho^{k}_H(r;c_1,\cdots, c_{s} )<0$ implies that
$W_{H}^{k}(r;c_1,\cdots, c_{s} )$ is empty, and whether
$\rho^{k}_H(r;c_1,\cdots, c_{s} )\ge 0$ implies that
$W_{H}^{k}(r;c_1,\cdots, c_{s} )$ is non-empty of dimension
$\rho^{k}_H(r;c_1,\cdots, c_{s} )$. We end the section giving
examples of situations where the expected dimension
$\rho^{k}_H(r;c_1,\cdots, c_{s} )<  0$ and
$W_{H}^{k}(r;c_1,\cdots, c_{s} )$ is non-empty; and examples where
$\rho^{k}_H(r;c_1,\cdots, c_{s} )>0$ and $W_{H}^{k}(r;c_1,\cdots,
c_{s} )$ is non-empty of dimension strictly greater than
$\rho^{k}_H(r;c_1,\cdots, c_{s} )$, in contrast with the classical
case. In section 3, we will analyze how the Brill-Noether
stratification will allow us to compare moduli spaces of vector
bundles on a smooth projective surface stable with respect to
different polarizations. It turns out that the ample cone of a
smooth projective surface $X$ has a chamber structure such that
the moduli space $M_{X,H}(r;c_1,c_2)$ only depends on the chamber
of $H$ and the problem consists on determining how the moduli
space changes when the polarization crosses a wall between two
chambers (see Definition \ref{parets}). Very surprisingly we will
realize that in many situations, the Brill Noether locus  controls
these changes.  As a by-product we will obtain a huge family of
examples where the dimension of the Brill-Noether loci coincide
with the expected one. In the last section, we describe the
Brill-Noether loci in the moduli space of mathematical instanton
bundles.

\vskip 2mm  \noindent {\bf Notation:} We will work over an
algebraically closed field $K$ of characteristic zero. Let $X$ be
a smooth projective variety of dimension $n$ and let $E$ be a rank
$r$ vector bundle on $X$ with Chern classes $c_i(E)=c_i$, $1 \leq
i \leq s:= min\{r,n\}$. Set $\chi(r;c_1, \cdots, c_{s}):=\chi(E)$.
We will write $h^i(E)$ (resp. $ext^i(E,F)$) to denote the
dimension of the $i$-th cohomology group $H^i(X,E)=H^i(E)$ (resp.
$i$-th Ext group $Ext^i(E,F)$) as a $K$-vector space. The sheaf
$K_X$ will denote the canonical sheaf on $X$.

\section{Construction of the Brill-Noether Loci}

The goal of this section is to prove the existence of a
Brill-Noether type stratification on the moduli space of stable
vector bundles on smooth projective varieties analogous to the
classical stratification of the Picard variety $Pic^d(C)$ of
degree $d$ line bundles on a smooth projective curve $C$.

\vspace{3mm} Let us start fixing the notation and some basic
definitions. We consider  $X$ an $n$-dimensional  smooth
projective variety, $H$ an ample divisor on $X$, $r \geq 2$ an
integer and $c_i \in H^{2i}(X,\ZZ)$ for $i=1, \cdots,
s=min\{r,n\}$. We denote by $M_{H}=M_{X,H}(r;c_1, \cdots, c_{s})$
the moduli space of rank $r$, vector bundles $E$ on  $X$, with
fixed Chern classes $c_i(E)=c_i$ for $i=1, \cdots, s$, and
$H$-stable according to the following definition due to Mumford
and Takemoto.

\begin{definition} Let $H$ be an ample line bundle on a
smooth projective $n$ dimensional variety $X$. For a torsion free
sheaf $F$ on $X$ we set
$$ \mu(F)=\mu_H (F):=\frac{c_1(F)H^{n-1}}{rk(F)}.$$ The sheaf $F$ is said to be
{\em $H$-semistable} if
$$ \mu_H (E)\le \mu_H (F)
$$ \noindent for all non-zero subsheaves $E\subset F$ with
$rk(E)<rk(F)$; if strict inequality holds then $F$ is {\em
$H$-stable}. Notice that for rank $r$ vector bundles $F$ on $X$
with $(c_1(F)H^{n-1},r)=1$, the concepts of $H$-stability and
$H$-semistability coincide.
\end{definition}

\begin{remark}
\label{estabilitatdepen} The definition of stability depends on
the choice of the ample line bundle $H$. The changes of the moduli
space that occur when the line bundle $H$ varies have been studied
by several people in great detail and reveals interesting
phenomena (see for instance \cite{Qin93}; \cite{JPAA};
\cite{crelle}; \cite{Michigan}; \cite{nagoya} and references
therein). In section 3, we will see how the Brill-Noether loci
allow us to study these changes.
\end{remark}

\vspace{3mm}

The main goal of this section is to construct a subvariety
$W_{H}^{k}(r;c_1,\cdots, c_{s} )$ of $M_H$ whose support is the
set of rank $r$, $H$-stable vector bundles $E$ on $X$ with Chern
classes $c_i$ such that $h^0(E) \geq k$. In other words, we are
going to construct a variety
$$W_{H}^{k}(r;c_1,\cdots, c_{s} )$$ such that
\[Supp(W_{H}^{k}(r;c_1,\cdots, c_{s} ))= \{ E \in M_H | h^0(E) \geq k \}. \]

\vspace{3mm}

To achieve our propose, we first need to recall the definition of
$k$-th determinantal variety. Let $\phi:E \rightarrow F$ be a
morphism between locally free sheaves of ranks $e$, $f$ over an
algebraic variety $X$. Upon choosing local trivializations of $E$
and $F$ over an open set $U \subset X$, the morphism $\phi$ is
represented by an $e \times f$ matrix $A$ of holomorphic
functions. We denote by $U_k$ the subset of $U$ whose ideal is
generated by the $(k+1)\times (k+1)$ minors of $A$. It is easy to
see that $U_k$ does not depend on the choice of the
trivialization, and therefore there is a well-defined subvariety
$X_k(\phi)$ of $X$ such that
\[ X_k(\phi) \cap U=U_k \]
for every open set $U \subset X$. The variety $X_k(\phi)$ is
called the {\bf $k$-th determinantal  variety} or the {\bf $k$-th
degeneracy locus} of $\phi$; it is supported on the set
\[ \{p \in X| rk(\phi_p) \leq k \} \]
and it is clear from the definition that $X_k(\phi)$ has
codimension at most $(e-k)(f-k)$ when it is non-empty and
$$X_k (\phi)\subset Sing(X_{k+1}(\phi))$$
whenever $X_{k+1}(\phi) \neq X$.

\vspace{3mm}

We are now ready to define the Brill-Noether filtration of
$M_{X,H}(r;c_1, \cdots, c_{s})$ and to give a formula for the
expected dimension of the Brill-Noether locus.

\vspace{3mm}

\begin{theorem} \label{construccio} Let $X$ be a smooth projective variety of dimension $n$ and
consider a moduli space  $M_{H}=M_{X,H}(r;c_1, \cdots, c_{s})$ of
rank $r$, $H$-stable vector bundles $E$ on $X$ with fixed Chern
classes $c_i(E)=c_i$. Assume that for any $E \in M_H$, $H^i(E)=0$
for $i \geq 2$. Then, for any $k \geq 0$, there exists a
determinantal variety $W_{H}^{k}(r;c_1,\cdots, c_{s} )$ such that
\[Supp(W_{H}^{k}(r;c_1,\cdots, c_{s} ))= \{ E \in M_H | h^0(E) \geq k \}. \]
Moreover, each non-empty irreducible component of
$W_{H}^{k}(r;c_1,\cdots, c_{s} )$ has dimension at least
$$\dim(M_H) -k(k-\chi(r;c_1, \cdots, c_s)),$$ and
$$W_{H}^{k+1}(r;c_1,\cdots, c_{s} )\subset Sing(W_{H}^{k}(r;c_1,\cdots, c_{s}
))$$ whenever $ W_{H}^{k}(r;c_1,\cdots, c_{s} ) \neq
M_{X,H}(r;c_1, \cdots, c_{s})$.
\end{theorem}
\begin{proof} First of all assume that $M_H$ is
a fine moduli space. Let $\cU \rightarrow X \times M_{H}$ be a
universal family such that for any $t \in M_H$, $\cU|_{X\times \{
t\}} =E_t$ is an $H$-stable rank $r$ vector bundle on $X$ with
Chern classes $c_i(E_t)=c_i$. Let $D$ be an effective divisor on
$X$ such that for any $t \in M_H$,
\begin{equation} \label{D} h^0(E_t(D))= \chi(E_t(D)), \quad
H^i(E_t(D))=0, \quad i \geq 1. \end{equation} Consider $\cD=D
\times M_H$ the corresponding product divisor on $X \times M_H$
and denote by
 \[\nu: X \times M_H \rightarrow M_H \]
 the natural projection. It follows from (\ref{D}) and the base
 change theorem that $\nu_*\cU(\cD)$ is a locally free sheaf of
 rank $\chi(E_t(D))$ on $M_H$ and
 \[ R^i\nu_*\cU(\cD)=0, \quad i>0.\]
 Therefore, applying the functor $\nu_*$ to the short exact
 sequence
 \[ 0 \rightarrow \cU \rightarrow \cU(\cD) \rightarrow \cU(\cD)/\cU \rightarrow 0\]
 we get the following exact sequence
  \[ 0 \rightarrow \nu_*\cU \rightarrow \nu_*\cU(\cD) \mapright{\gamma} \nu_*(\cU(\cD)/\cU) \rightarrow R^1\nu_*\cU \rightarrow 0.\]
The map $\gamma$ is a morphism between locally free sheaves on
$M_H$ of rank $\chi(E_t(D))$ and $\chi(E_t(D))-\chi(E)$
respectively and the $(\chi(E_t(D))-k)$-th determinantal variety
$$W^k_{H}(r;c_1,\cdots, c_{s} ) \subset M_H$$ associated to it has
support
\[ \{E_t \in M_H| \rank \gamma_{E_t} \leq \chi(E_t(D))-k \} \]
i.e. $W^k_{H}(r;c_1,\cdots, c_{s} )$ is the locus where the fiber
of $R^1\nu_*\cU$ has dimension at least
$(\chi(E_t(D))-\chi(E_t))-(\chi(E_t(D))-k)=k-\chi(E_t)$. For any
$E_t \in M_H$ the assumption $h^i(E_t)=0$, $i \geq 2$, implies \[
h^1(E_t)= h^0(E_t)-\chi(E_t).\] Thus,
\[\begin{array}{ll}Supp(W_{H}^{k}(r;c_1,\cdots, c_{s} )) & = \{ E \in M_H | h^1(E) \geq k-\chi(E) \}
\\ & = \{ E \in M_H | h^0(E) \geq k \}. \end{array} \] Using the language of Fitting ideals, we have that
$W_{H}^k(r;c_1,\cdots, c_{s} )$ is the subvariety of $M_{H}$
defined by the $(k-\chi(E_t))$th Fitting ideal of $R^1\nu_* \cU$.
Moreover, it can also be seen that $W_{H}^k(r;c_1,\cdots, c_{s} )$
represents the functor
\[S \rightarrow \bigg \{ \begin{array}{l} \mbox{equivalence classes of families
$\cF$ on $S \times M_H \mapright{\nu} M_H$} \\
\mbox{of $H$-stable rank $r$ vector bundles $E$ on $S$ with
$c_i(E)=c_i$} \\
\mbox{such that the Fitting rank of $R^1\nu_*\cF$ is at least
$(k-\chi(E))$}
\end{array} \bigg \}.
\]
Finally, since $W^k_{H}(r;c_1,\cdots, c_{s} )$ is a
$(\chi(E_t(D))-k)$-determinantal variety associated to a morphism
between locally free sheaves of rank $\chi(E_t(D))$ and
$\chi(E_t(D))-\chi(E)$ respectively, any of its non-empty
irreducible components has dimension greater or equal to
$\dim(M_H) -k(k-\chi(E))$ and
$$W_{H}^{k+1}(r;c_1,\cdots, c_{s} )\subset Sing(W_{H}^{k}(r;c_1,\cdots, c_{s}
))$$ whenever $ W_{H}^{k}(r;c_1,\cdots, c_{s} ) \neq
M_{X,H}(r;c_1, \cdots, c_{s})$.

If $M_H$ is not a fine moduli space, it is also possible to carry
out the construction of the Brill-Noether locus using only the
local existence of a universal sheaf on $M_H$. Indeed, we carry
out the constructions locally, we show the independence of the
choice of the locally universal sheaf and we conclude that the
construction glue as a global algebraic object.
\end{proof}

\begin{remark}
(1) The cohomological assumptions in Theorem \ref{construccio} are
natural if we want to have a filtration of the moduli space $M_H$
by the subvarieties $W_{H}^k(r;c_1,\cdots, c_{s} )$. Indeed, if
$X$ is an $n$-dimensional projective variety, then any vector
bundle $E$ on $X$ has $n+1$ cohomological groups whose dimensions
are related by Riemann-Roch theorem and one is forced to look for
a multigraded filtration of the moduli space $M_H$ by means of the
sets $\{ E \in M_H| h^i(E) \geq k_i\}$. Under the assumptions of
Theorem \ref{construccio}, $h^i(E)=0$ for $i \geq 2 $ and for any
$E \in M_H$, the only non-vanishing cohomology groups are $H^0(E)$
and $H^1(E)$ and their dimensions are subject to one relation
given by Riemann-Roch theorem: $\dim H^0(E)-\dim
H^1(E)=\chi(E)=\chi(r;c_1, \cdots, c_s)$. Hence, it makes sense to
consider only the filtration of the moduli space $M_H$ by the
dimension of the space of global sections.

(2) We want to point out that there exists plenty of vector
bundles satisfying the cohomological conditions of Theorem
\ref{construccio}. For instance, instanton bundles on
$\PP^{2n+1}$, Schwarzenberger bundles on $\PP^n$, Steiner bundles
on $\PP^n$, Steiner and Spinor bundles on a hyperquadric $Q_n
\subset \PP^{n+1}$, etc.
\end{remark}

\begin{definition} The variety $W_{H}^k(r;c_1,\cdots, c_{s} )$ is called the {\bf $k$-Brill-Noether
locus} of the moduli space $M_{H}$ (or simply Brill-Noether locus
if there is no confusion) and  $$\rho_{H}^k(r;c_1,\cdots, c_{s}
):= \dim M_H-k(k-\chi(r;c_1,\cdots, c_{s} ))$$ is called the {\bf
generalized Brill-Noether number}.

By Theorem \ref{construccio}, the Brill-Noether locus
$W_{H}^k(r;c_1,\cdots, c_{s} )$ has dimension greater or equal to
$\rho_{H}^k(r;c_1,\cdots, c_{s} )$ and the number
$\rho_{H}^k(r;c_1,\cdots, c_{s} )$ is also called the expected
dimension of the corresponding Brill-Noether locus. Hence,  we are
led to pose the question whether the dimension of the
Brill-Noether locus $W_{H}^k(r;c_1,\cdots, c_{s} )$ and its
expected dimension coincide provided the Brill-Noether locus
$W_{H}^k(r;c_1,\cdots, c_{s} )$ is non-empty.
\end{definition}

\begin{notation} If there is no confusion then, we will simply write $W^k$ and $\rho^k$ instead of
$W_{H}^k(r;c_1,\cdots, c_{s} )$ and $\rho_{H}^k(r;c_1,\cdots,
c_{s} )$.

 We will say that the
Brill-Noether locus is defined in the moduli space $M_{H}$
whenever the assumptions of Theorem \ref{construccio} are
satisfied, i.e. for any $E \in M_H$, $H^i(E)=0$ for $i \geq 2$.
\end{notation}

\begin{remark}
Notice that when $X$ is a smooth projective curve and we consider
the moduli space $Pic^d(X)$ of degree $d$ line bundles on $X$,
then we recover the classical Brill-Noether loci which have been
well known since the last century, and the generalized
Brill-Noether number is the classical Brill-Noether number
$\rho=\rho(g,r,d)=g-(r+1)(g-d+r)$.
\end{remark}

\begin{corollary} \label{superficie} Let $X$ be a smooth projective surface and
let   $M_{H}=M_{X,H}(r;c_1,c_2)$ be a moduli space  of rank $r$,
$H$-stable vector bundles $E$ on $X$ with fixed Chern classes
$c_i(E)=c_i$. Assume that $c_1H \geq r K_X H$.  Then, for any $k
\geq 0$, there exists a determinantal variety
$W_{H}^{k}(r;c_1,c_2)$ such that
\[Supp(W_{H}^{k}(r;c_1,c_2))= \{ E \in M_H | h^0(E) \geq k \}. \]
Moreover, each non-empty irreducible component of
$W_{H}^{k}(r;c_1,c_2)$ has dimension greater or equal to
$$\rho_H^k(r;c_1,c_2)=\dim(M_H) -k(k-r(1+P_a(X))+\frac{c_1K_X}{2}-\frac{c_1^2}{2}+c_2)$$
and $W_{H}^{k+1}(r;c_1,c_2) \subset Sing(W_{H}^{k}(r;c_1,c_2))$
whenever $W_{H}^{k}(r;c_1,c_2) \neq M_{X,H}(r;c_1,c_2)$.
\end{corollary}
\begin{proof} First of all notice that for any $E \in M_H$, the
numerical condition $c_1(E)H > rK_X H$ implies that $H^2(E)=0$.
Indeed,  by Serre duality we have:
\[H^2(E) \cong H^0(E^*(K_X)) \] and since $E$ is an $H$-stable vector bundle on $X$,
$E^*$ is also $H$-stable. Thus, if $H^2(E) \neq 0$, $\cO_{X}(-K_X)
\hookrightarrow E^*$ and since $E^*$ is $H$-stable we get
\[(-K_X H)< \frac{c_1(E^*)H}{r}=-\frac{c_1(E)H}{r},\]
which contradicts the assumption $c_1(E)H \geq rK_X H$. Hence the
result follows from Theorem \ref{construccio} and the fact that by
the Riemann-Roch theorem \[
\chi(r;c_1,c_2)=r(1+P_a(X))-\frac{c_1K_X}{2}+\frac{c_1^2}{2}-c_2.\]
\end{proof}

\vspace{3mm}

For any sheaf $E$ on $\PP^2$, denote by $\chi^+=\chi^+(E):= \max
\{\chi(E), 0 \}$. In \cite{GoHi}, G\"{o}ttsche and Hirschowitz
gave, under the assumption $c_1>-3r$, a lower bound for the
codimension of the Brill-Noether loci $W^{\chi^++1}(r;c_1,c_2)$ of
rank $r$, stable vector bundles $E$ on $\PP^2$ with fixed Chern
classes $c_i(E)=c_i$ such that
 $h^0(E) \geq \chi^++1$. From the previous result we
get the following upper bound:

\begin{corollary} Let $W^{\chi^++1}(r;c_1,c_2)$ be the Brill-Noether locus  of rank $r$, stable vector bundles $E$
on $\PP^2$ with fixed Chern classes $c_i(E)=c_i$ such that
 $h^0(E) \geq \chi^++1$. Assume that $c_1>-3r$. Then, the following holds:

\noindent(a) If $\chi^+= \chi(r;c_1,c_2)>0$
\[2 \leq \codim(W^{\chi^++1}(r;c_1,c_2)) \leq \chi(r;c_1,c_2)+1. \]
(b) If $\chi^+= \chi(r;c_1,c_2)=0$
\[ \codim(W^{\chi^++1}(r;c_1,c_2))=1. \]
(c) If $\chi^+=0> \chi(r;c_1,c_2)$
\[ \codim(W^{\chi^++1}(r;c_1,c_2)) \leq (\chi^++1)(\chi^++1-\chi(r;c_1,c_2)). \]
\end{corollary}
\begin{proof} The lower bounds follow from \cite{GoHi}; Theorem 1.
Since $K_{\PP^2}=\cO_{\PP^2}(-3)$, the hypothesis $c_1>-3r$ is
equivalent to $c_1(E)H > rK_X H$ and the upper bounds follow from
Corollary \ref{superficie}.
\end{proof}

\vspace{3mm}

In subsequent sections we will see that there are plenty of
situations where the assumptions of Theorem \ref{construccio} are
satisfied and we will  prove that in  several of them the
Brill-Noether loci have exactly the expected dimension, showing
 that the bound given in Theorem \ref{construccio} is
sharp.

\vspace{3mm}

Once we have proved the existence of these varieties, it is
natural to ask whether the condition $\rho_H^k(r;c_1,\cdots, c_{s}
)<0$ implies that the variety $W_{H}^k(r;c_1,\cdots, c_{s} )$ is
empty and whether the condition $\rho_H^k(r;c_1,\cdots, c_{s} )
\geq 0$ implies that the variety $W_{H}^k(r;c_1,\cdots, c_{s} )$
is non-empty. Indeed we are led to pose the following three
questions.

\begin{question} \label{mainquestion} Let $X$ be a smooth projective variety of
dimension $n$. We consider a moduli space $M_{H}(r;c_1,\cdots,
c_{s} )$  of rank $r$, $H$-stable vector bundles on $X$ where the
Brill-Noether locus is defined.

\begin{itemize}
\item[(1)] Whether $ \rho_H^k(r;c_1,\cdots, c_{s} )<0$ implies
$W_{H}^k(r;c_1,\cdots, c_{s} ) =\emptyset$ ? \item[(2)]Whether $
\rho_H^k(r;c_1,\cdots, c_{s} )  \geq 0$ implies
$W_{H}^k(r;c_1,\cdots, c_{s} ) \neq \emptyset$ ? \item[(3)]Whether
$ \rho_H^k(r;c_1,\cdots, c_{s} )  \geq 0$ and
$W_{H}^k(r;c_1,\cdots, c_{s} ) \neq \emptyset$ implies
 $$ \rho_H^k(r;c_1,\cdots, c_{s} )=\dim W_{H}^k(r;c_1,\cdots, c_{s}
 ) \quad ?$$
\end{itemize}
\end{question}

If $C$ is a smooth algebraic curve and the moduli space is the
Picard variety $Pic^d(C)$ of degree $d$ line bundles on $C$, the
answer to  Question \ref{mainquestion} is well-known. In fact,
classical Brill-Noether theory has its roots dating more than a
century ago and it is concerned with the subvarieties of the
Picard variety $Pic^d(C)$ determined by  degree $d$ line bundles
on $C$ having at least a specified number of independent sections.
Basic questions, concerning non-emptiness, connectedness,
irreducibility, dimension, singularities, cohomology classes, etc
... have been completely answered when the underlying curve is a
generic curve in the moduli space $M_g$ of curves of genus $g$.
Indeed, the Brill-Noether locus is non-empty if $\rho \geq 0$ and
connected if $\rho >0$. For a generic curve in $M_g$, the
Brill-Noether locus is empty if $\rho <0$ and is irreducible of
dimension $\rho$ if $\rho>0$. Modern proofs of these results have
been given by Kempf, Kleiman and Laksov, Fulton and Lazarsfeld,
Griffiths and Harris, and Gieseker and a full treatment of this is
contained in \cite{ACGH}. The Brill-Noether loci in the moduli
space of vector bundles of higher rank on a generic curve $C$ has
been studied by Teixidor and others in a series of papers in
1994-2007.

\vspace{3mm}

Next example shows that if we deal with algebraic varieties of
dimension greater than one, Question \ref{mainquestion} (1) is no
longer true; it gives an example of negative generalized
Brill-Noether number and the corresponding Brill-Noether locus is
non-empty, in contrast with the classical case. Even more, we give
examples where the expected dimension of the Brill-Noether locus
and the dimension of the Brill-Noether locus do not coincide in
spite of being positive the generalized Brill-Noether number.
Indeed, we have

 \vspace{3mm}
\begin{example} Let $X=\PP^1 \times \PP^1$ be a quadric surface in $\PP^3$.
We denote by $l_1$, $l_2$ the generators of $Pic(X)$ and for any
integer $n \geq 2$ we fix the ample line bundle $L=l_1+nl_2$. We
will describe the Brill-Noether stratification of $M_{X,L}(2;
(2n-1)l_2, 2n)$. Indeed,  since $(2n-1)l_2 L=2n-1> -4n-4=2K_X L$,
its existence is guaranteed  by Corollary \ref{superficie}. Let us
now prove that for any integer $n \geq 2$ and any integer $j$, $0
\leq j \leq n$, the Brill-Noether locus $W_L^{j}(2;(2n-1)l_2,2n)$
is non-empty. To this end,  we consider $\cF$ the irreducible
family parameterizing rank two vector bundles $E$ on $X$ given by
an exact sequence
 \begin{equation}
 \label{rk2se5} \hspace{6mm}
 0 \rightarrow  \cO_{X} \rightarrow   E  \rightarrow
 \cO_{X}((2n-1)l_2)\otimes I_{Z} \rightarrow 0   \end{equation}
where $Z$ is a locally complete intersection 0-dimensional scheme
of length $2n$ such that $H^0I_{Z}((2n-1)l_2)=0$.

\vspace{2mm} Notice that since $|Z|=2n$ and
$h^0\cO_X((2n-1)l_2)=2n$, the condition
\[ H^0I_{Z}((2n-1)l_2)=0 \] is satisfied for all generic $Z \in Hilb^{2n}(X)$
and $\cF$ is non-empty. In addition, it can be seen that $\dim
\cF= 4(2n)-3$ and that $ \cF \hookrightarrow M_{X,L}
(2;(2n-1)l_2,2n)$ (see \cite{nagoya}; Proposition 4.6).

\vskip 2mm \noindent{\bf Claim 1:} $W^1_L(2;(2n-1)l_2,2n) \cong
\cF$.

\vskip 2mm  \noindent{\bf Proof of Claim 1:} Any $E \in \cF$ is
$L$-stable and since $H^0I_{Z}((2n-1)l_2)=0$, we have $h^0(E)=1$.
Thus $\cF \subset W^1_L(2;(2n-1)l_2,2n)$. Let us prove the
converse. We take a vector bundle
 $E \in W^1_L(2;(2n-1)l_2,2n)$ and     a non-zero global
section  $s$ of $E$. We denote by $Y$ its scheme of zeros and by
$D$ the maximal effective divisor contained in $Y$. Then $s$ can
be regarded as a section of $E(-D)$ and its scheme of zeros has
codimension $\geq 2$. Thus, for some effective divisor
$D=al_1+bl_2$ we have a short exact sequence
\begin{equation}
 \label{} \hspace{6mm}
 0 \rightarrow  \cO_{X}(D) \rightarrow   E  \rightarrow
 \cO_{X}((2n-1)l_2-D)\otimes I_{Z} \rightarrow 0   \end{equation}
where $Z$ is a locally complete intersection 0-cycle. Since $D$ is
effective, $a,b \geq 0$ and by the $L$-stability of $E$ we have
\[(al_1+bl_2)L = (an+b) < \frac{2n-1}{2}=\frac{c_1(E)L}{2}. \]
Therefore, $a=b=0$ and in fact $E \in \cF$.

It follows from Claim 1 that the Brill-Noether locus
$W^1_L(2;(2n-1)l_2,2n)$ is an irreducible variety of dimension
$8n-3$ and notice that in this case, its dimension coincides with
the expected one. Indeed,

\[\begin{array}{ll}
\rho_L^1(2;(2n-1)l_2,2n) &= \dim M_{X,L}(2;(2n-1)l_2,2n)
-1(1-\chi(2;(2n-1)l_2,2n)) \\ & = 8n-3 \\
 \end{array} \]
 since by Riemann-Roch theorem,
 \begin{equation} \label{RR} \chi(2;(2n-1)l_2,2n)=2+\frac{((2n-1)l_2)(2l_1+2l_2)}{2}+ \frac{((2n-1)l_2)^2}{2}-2n=1.
 \end{equation}

For any $i$, $1 \leq i \leq n-1$, we can choose a $0$-dimensional
scheme $Z_i$ on $X$ of length $2n$ such that
$h^0(I_{Z_i}((2n-1)l_2))=i$ and $h^0(I_{Z_i}((2n-i-1)l_2))=0$.
Thus, if we denote by $E_i$ the rank two vector bundle on $X$
given by the exact sequence
\begin{equation}
 \label{} \hspace{6mm}
 0 \rightarrow  \cO_{X} \rightarrow   E_i  \rightarrow
 \cO_{X}((2n-1)l_2)\otimes I_{Z_i} \rightarrow 0   \end{equation}
we have $h^0(E_i)=i+1$.

\vskip 2mm \noindent{\bf Claim 2:} $E_i$ is $L$-stable.

\vskip 2mm \noindent{\bf Proof of Claim 2:} Since $E_i$ is given
by a non-trivial extension
\begin{equation}
 \label{} \hspace{6mm}
 0 \rightarrow  \cO_{X} \rightarrow   E_i  \rightarrow
 \cO_{X}((2n-1)l_2)\otimes I_{Z_i} \rightarrow 0,   \end{equation}
given a sub-line bundle $\cO_X(al_1+bl_2)$ of $E_i$ we have two
possible cases:
\[ (1) \quad \cO_X(al_1+bl_2) \hookrightarrow \cO_X \]
\[ (2) \quad \cO_X(al_1+bl_2) \hookrightarrow \cO_{X}((2n-1)l_2)\otimes I_{Z_i}. \]
In case $(1)$, $(al_1+bl_2)L \leq 0 <
\frac{2n-1}{2}=\frac{c_1(E_i)L}{2}$. In case $(2)$,
$-al_1+(2n-1-b)l_2$ is an effective divisor and hence $a \leq 0$
and $b \leq 2n-1$. On the other hand, since
\[H^0(\cO_X(al_1+(b-i)l_2)) \subset  H^0(I_{Z_i}((2n-i-1)l_2))=0 \]
we have $a <0$ or $b <i$. If $b<i$, since $a \leq 0$  and $i\le
n-1$, we have
$$(al_1+bl_2)L =an+b \le n-1< \frac{2n-1}{2} =\frac{c_1(E_i)L}{2}.$$
Assume $a<0$. In that case, since $b \leq 2n-1$ we get
$$(al_1+bl_2)L =an+b \leq -n+b \leq n-1 < \frac{2n-1}{2} =\frac{c_1(E_i)L}{2}.$$
Therefore, $E_i$ is $L$-stable which proves Claim 2.

\vskip 2mm By Claim 2, $E_i \in W^{i+1}_L(2;(2n-1)l_2,2n)$ and we
get a chain of non-empty Brill-Noether loci
\[M_{X,L}(2;(2n-1)l_2,2n) \supset W^1_L(2;(2n-1)l_2,2n) \supset W^2_L(2;(2n-1)l_2,2n)\supset \cdots \hspace{25mm} \]
 \[\hspace{60mm} \cdots
\supset  W^{n}_L(2;(2n-1)l_2,2n) \supsetneq  \emptyset. \]

Notice that for any $k$, $1\le k \le n$, such that $8n-3 <k(k-1)$
the Brill-Noether locus $W^{k}_L(2;(2n-1)l_2,2n)$ is non-empty and
the generalized Brill-Noether number, $\rho^k_L(2;(2n-1)l_2,2n)$,
is negative. In fact,
\[
\begin{array}{ll}\rho^k_L(2;(2n-1)l_2,2n) & = \dim
M_{X,L}(2;(2n-1)l_2,2n)-k(k-\chi(2;(2n-1)l_2,2n)) \\
& =8n-3-k(k-1) \end{array}\] where the last equality follows from
Proposition \ref{moduli} and the equation (\ref{RR}).

Finally, we have that for any $k$ such that $2<k(k-1)<8n-3$, the
generalized Brill-Noether number, $\rho^k_L(2;(2n-1)l_2,2n)$, is
positive; however, the Brill-Noether locus
$W^{k}_L(2;(2n-1)l_2,2n)$ is non-empty and its dimension is
greater than the expected one. In fact, it is enough to observe
that
$$\dim W^{k}_L(2;(2n-1)l_2,2n)=8n-2k-1>\rho^k_L(2;(2n-1)l_2,2n).$$

\end{example}

\section{Brill-Noether loci and change of polarizations}

In this section we will see that the Brill-Noether stratification
allow us to compare moduli spaces of vector bundles on smooth
projective surfaces, stable with respect to different
polarizations.

\vspace{3mm}

Let $X$ be a smooth projective variety of dimension $n$. As we
pointed out in Remark \ref{estabilitatdepen},  the notion of
stability of vector bundles on $X$ strongly depends on the ample
divisor. Hence it is natural to consider the following interesting
problem:

\begin{problem}
What is the difference between the moduli spaces
$$M_H=M_{X,H}(r;c_1, \cdots, c_{min\{r,n\}}) \quad \mbox{and}
\quad  M_L=M_{X,L}(r;c_1, \cdots, c_{min\{r,n\}})$$ where $H$ and
$L$ are two different polarizations?
\end{problem}

It turns out that the ample cone of $X$ has a chamber structure
such that the moduli space $M_{X,H}(r;c_1, \cdots,
c_{min\{r,n\}})$ only depends on the chamber of $H$ and the
problem consists on determining how the moduli space changes when
the polarization crosses a wall between two chambers (see
Definition \ref{parets}). These changes have been explicitly
described in very few occasions (see for instance \cite{JPAA},
\cite{Michigan}, \cite{nagoya}).

\vspace{3mm}

 We will focus our attention in  case
where $X$ is a Hirzebruch surface and we will deal with stable
rank two vector bundles on $X$. Very surprisingly we will realize
that in a huge family of moduli spaces, the difference between two
moduli spaces $M_{L}(2;c_1,c_2)$ and $M_{H}(2;c_1,c_2)$ (where $L$
and $H$ are two polarizations sitting in chambers sharing a common
wall), is precisely described by suitable Brill-Noether loci.
Moreover, we will be able to  compute the dimension of these
Brill-Noether loci  and we will prove that the expected dimension
of these Brill-Noether loci is indeed the dimension.

\vspace{3mm}

To start with,  let us recall the basic results about walls and
chambers due to Qin (\cite{Qin93}).

\vspace{3mm}

\begin{definition} \label{parets}
 (i) Let  $\xi \in Num(X)
\otimes \RR$. We define
\[ \cW^{\xi} := C_X \cap \{ x \in Num(X) \otimes \RR | x \xi =0 \}. \]
\noindent (ii) Define ${ \cW}(c_1,c_2)$ as the set whose elements
consist of $\cW^{\xi}$, where  $\xi$ is the numerical equivalence
class of a divisor $D$ on $X$ such that $\cO_X(D+c_1)$ is
divisible by 2 in $Pic(X)$, and that
\[ D^2  <0 ;\hspace{8mm} c_2 + \frac{D^2-c_1^2}{4}= [Z] \]
for some locally complete intersection codimension-two cycle $Z$
in $X$.

\noindent (iii) A wall of type $(c_1,c_2)$ is an element in ${
\cW}(c_1,c_2)$. A chamber of type $(c_1,c_2)$ is a connected
component of $C_X \setminus{ \cW}(c_1,c_2)$. A $\ZZ$-chamber of
type $(c_1,c_2)$ is the intersection of $Num(X)$ with some chamber
of type $(c_1,c_2)$.

\noindent (iv) A face of type $(c_1,c_2)$ is ${ F}= \cW^{\xi} \cap
{\overline C}$, where $\cW^{\xi}$ is a wall of type $(c_1,c_2)$
and $C$ is a chamber of type $(c_1,c_2)$.
\end{definition}

\vspace{4mm} We say that a wall $\cW^{\xi}$ of type $(c_1,c_2)$
separates two polarizations $L$ and $L'$ if, and only if, $\xi  L
<0< \xi  L'$.

\vspace{4mm}
\begin{remark}
\label{nomes:depen} In \cite{Qin93}; Corollary 2.2.2 and Remark
2.2.6, Qin proves that the moduli space $M_{X,L}(2;c_1,c_2)$ only
depends on the chamber of $L$ and that the study of moduli spaces
of rank two vector bundles stable with respect to a polarization
lying on walls may be reduced to the study of moduli spaces of
rank two vector bundles stable with respect to a polarization
lying on $\ZZ$-chambers.
\end{remark}

\vspace{4mm}
\begin{definition}
\label{E:xi}  Let $\xi $ be a numerical equivalence class defining
a wall of type $(c_{1},c_{2})$. We define $\cE_{\xi
}(c_{1},c_{2})$ as the quasi-projective variety parameterizing
rank 2 vector bundles $E$ on $X$ given by an extension
\[ 0 \rightarrow  \cO_X(D) \rightarrow   E  \rightarrow
 \cO_X(c_{1}-D)\otimes I_{Z} \rightarrow 0  \]
where $D$ is a divisor with $2D-c_{1}\equiv \xi $ and $Z$ is a
locally complete intersection 0-cycle of length $c_{2}+(\xi
^{2}-c_{1}^{2})/4$. Moreover, we require that $E$ is not given by
the trivial extension when $\xi ^{2}=c_{1}^{2}-4c_{2}$.
\end{definition}

\vspace{3mm}
\begin{remark}
\label{Qin}  By \cite{Qin93}; Theorem 1.2.5, if $L_{1}$ and
$L_{2}$ are two ample divisors on $X$ and $E$ is a rank 2 vector
bundle on $X$ which is $L_{1}$-stable but $L_{2}$-unstable, then
we have $E \in \cE_{\xi }(c_{1},c_{2})$ where $\xi $ defines a
non-empty wall of type $(c_{1},c_{2})$ separating $L_{1}$ and
$L_{2}$ (i.e. $\xi L_{1}<0<\xi L_{2}$; moreover, we can consider
the ample divisor $L:=(\xi L_{2})L_{1}-(\xi L_{1})L_{2}$ on $X$
and we have $L \xi =0$). More can be said, by \cite{Qin93};
Theorem 1.3.3, given $L_1$ and $L_2$ two polarizations lying on
chambers $C_1$ and $C_2$, sharing a common wall,
 we have
 \refstepcounter{equation} \hspace{4mm}
 \[  (\theequation) \label{ratse1} \hspace{4mm}
 M_{L_1}(2;c_{1},c_2)=(M_{L_2}(2;c_{1},c_2) \setminus  \sqcup_{\xi}
 \cE_{\xi }(c_{1},c_2)) \sqcup ( \sqcup_{\xi}\cE_{-\xi}(c_1,c_2)),
   \]
where $\xi$ satisfies $\xi L_1>0$ and runs over all numerical
equivalence classes which define the common wall $\cW$.
\end{remark}

\vspace{3mm}  In the next result, we have summarized  well-known
properties of some moduli spaces of stable rank two vector bundles
on projective surfaces that we will need later on (see for
instance \cite{JPAA}; Proposition 3.11).

\begin{proposition} \label{moduli} Let $X$ be a smooth, projective
rational surface with effective anticanonical line bundle and let
$H$ be an ample line bundle on $X$. Then, the moduli space
$M_{X,H}(2;c_1,c_2)$ of rank two, $H$-stable vector bundles $E$ on
$X$ with fixed Chern classes $c_i(E)=c_i$ is either empty or a
smooth irreducible variety of dimension
\[\dim M_{X,H}(2;c_1,c_2)= 4c_2- c_1^2-3. \]
\end{proposition}

\vspace{3mm} For any integer $e \geq 0$, let $X=\FF_e \cong
\PP(\cO_{\PP^1} \oplus \cO_{\PP^1}(-e))$ be a non singular,
Hirzebruch surface. We denote by $C_0$ and $F$ the standard basis
of $Pic(X) \cong \ZZ \oplus \ZZ$ such that $C_0^2=-e$,  $F^2=0$
and  $C_{0}F=1$. The canonical divisor is given by
          \[ K_{X}=-2C_0-(e+2)F \]
and it is well known that a divisor $L=aC_0+bF$ on $X$ is ample
if, and only if, it is very ample, if and only if, $a>0$ and
$b>ae$, and that $D=a'C_0+b'F$ is effective if and only if $a'
\geq 0$ and $b' \geq 0$ (\cite{hart}; V, Corollary 2.18).

\vspace{3mm}

Given an integer $c_2>0$ and $\alpha \in \{0,1 \}$, we denote by
$M_L(2;C_0+\alpha F,c_2)$ the moduli space of rank two, $L$-stable
vector bundles $E$ on $X$ with fixed Chern classes
$c_1(E)=C_0+\alpha F$ and $c_2(E)=c_2$. For any integer $n$, $1
\leq n \leq c_2-1$, consider the following ample divisor on $X$
\[ L_n:=C_0+(e+2c_2-\alpha-2n+1)F.\]
Notice that the equivalence class
\[ \xi_n:= C_0-(2c_2-\alpha-2n)F \]
defines a non-empty wall $\cW^{\xi_n}$ of type $(C_0+ \alpha
F,c_2)$ which separates the ample divisors $L_{n}$ and $L_{n+1}$.
Indeed, $\xi_{n}^2=-e-2(2c_2-\alpha-2n)<0$,
$\xi_n+c_1=2(C_0-(c_2-\alpha-n)F)$ is divisible by two in
$Pic(X)$, $c_2 + \frac{\xi_n^2-c_1^2}{4}=n>0$ and
\[ L_{n+1}\xi_n=-1<0< 1= L_{n}\xi_n.\]
We are led to pose the following problem

\begin{problem} Determine the difference between the moduli spaces
\[ M_{L_n}(2;C_0+\alpha F,c_2) \quad \mbox{and} \quad  M_{L_{n+1}}(2;C_0+\alpha F,c_2).\]
\end{problem}

\vspace{3mm}
\begin{remark}
Given an ample line bundle $L=aC_0 + bF$ on $X$, we can represent
$L$ as a point of coordinates $(a,b)$ in the plane. The following
picture gives us an idea of the situation we are discussing:

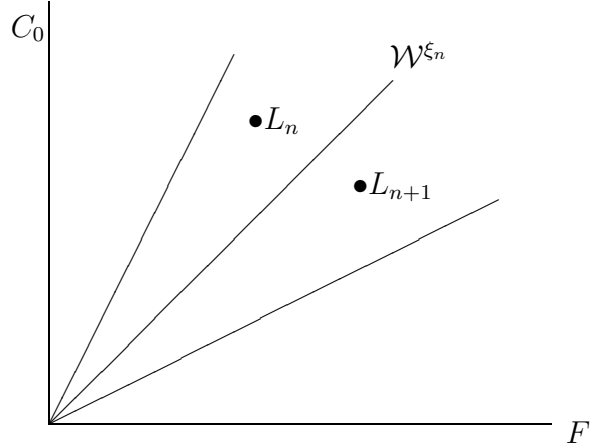
\vspace{20mm}

\begin{picture}(180,160)(-100,-10)
\put(0,0){\line(0,1){160}}
 \put(0,0){\line(1,2){70}}
  \put(0,0){\line(2,1){170}}
\put(0,0){\line(1,1){130}}
 \put(0,0){\line(1,0){190}}
\put(-8,150){\makebox(0,0){$C_0$}}
\put(85,115){\makebox(0,0){$\bullet L_n$}}
\put(140,140){\makebox(0,0){$\cW^{\xi_n}$}}
\put(130,90){\makebox(0,0){$\bullet L_{n+1} $}}
\put(200,-3){\makebox(0,0){$F$}}
\end{picture}

\vskip 5mm

\end{remark}

\vspace{3mm}

Next result completely solves this problem. We are going to prove
that these differences are surprisingly controlled by the
following two Brill-Noether loci: \[
W^1_{L_{n+1}}(2;\bar{c_1},n)\subset
M_{L_{n+1}}(2;\bar{c_1},n)\quad \mbox{and} \quad
W^1_{L_{n}}(2;\tilde{c_1},n) \subset M_{L_{n}}(2;\tilde{c_1},n)\]
where $\bar{c_1}=-C_0+(2c_2-\alpha -2n)F$ and
$\tilde{c_1}=C_0+(\alpha+2n-2c_2)F$.

\vspace{3mm}

\begin{theorem} Let $X=\FF_e$ be a smooth Hirzebruch surface, $c_2 >1$ an integer and $\alpha \in \{0,1 \}$. For any integer $n$, $1
\leq n \leq c_2-1$, consider the ample divisor
$L_n:=C_0+(e+2c_2-\alpha-2n+1)F$. Then
\[ M_{L_n}(2;C_0+\alpha F,c_2) \cong (M_{L_{n+1}}(2;C_0+\alpha F,c_2) \setminus W^1_{L_{n+1}}(2;\bar{c_1},n))
\sqcup W^1_{L_{n}}(2;\tilde{c_1},n).\]  Moreover, the
Brill-Noether loci \[ W^1_{L_{n+1}}(2;\bar{c_1},n)\subset
M_{L_{n+1}}(2;\bar{c_1},n)\quad \mbox{and} \quad
W^1_{L_{n}}(2;\tilde{c_1},n) \subset M_{L_{n}}(2;\tilde{c_1},n)\]
do have the expected dimension $\rho_{L_{n+1}}^1(2;\bar{c_1},n)$
and $\rho_{L_n}^1(2;\tilde{c_1},n)$, respectively.
\end{theorem}
\begin{proof} We have already seen that the numerical class
\[ \xi_n:= C_0-(2c_2-\alpha-2n)F \]
defines a non-empty wall $\cW^{\xi_n}$ of type $(C_0+ \alpha
F,c_2)$ which separates the ample divisors
$$L_{n}=C_0+(e+2c_2-\alpha-2n+1)F \quad \mbox{and}\quad  L_{n+1}=C_0+(e+2c_2-\alpha-2n-1)F.$$
Hence, by  Remark \ref{Qin} we have:
 \begin{equation}
  \label{des1} \hspace{4mm}
 M_{L_n}(2;C_0+ \alpha F,c_2)=(M_{L_{n+1}}(2;C_0+ \alpha F,c_2) \setminus  \sqcup_{\xi}
 \cE_{\xi }(c_{1},c_2)) \sqcup ( \sqcup_{\xi}\cE_{-\xi}(c_1,c_2)),
   \end{equation}
where $\xi$ satisfies $\xi L_n > 0$ and runs over all numerical
equivalence classes which define the common wall $\cW^{\xi_n}$
separating $L_{n}$ and $L_{n+1}$.

\noindent {\bf Claim 1:} $\xi_n$ is the only equivalence class
which defines the common wall $\cW^{\xi_n}$ and verifies $\xi_n
L_n> 0$.

\noindent {\bf Proof of Claim 1:} Let $L=aC_0+bF \in \cW^{\xi_n}$
be an ample divisor lying on the common wall defined by $\xi_n$.
By definition,
$$0=L  \xi_n=-ae-a(2c_2-\alpha-2n)+b$$ and thus
$L=aC_0+a(2c_2-\alpha-2n+e)$. Assume there exists a numerical
equivalence class $\xi=\sigma C_0+ \gamma F$ defining the common
wall $\cW^{\xi_n}$ and such that $\xi  L_n> 0$. In particular,
since $L=C_0+(2c_2-\alpha-2n+e) \in \cW^{\xi_n}$, we have
\[ 0=\xi L= (\sigma C_0+ \gamma F)(C_0+(2c_2-\alpha-2n+e)F)
=-\sigma e + \sigma (2c_2-\alpha-2n+e)+\gamma \] and  hence
$\gamma=-\sigma (2c_2-\alpha-2n)$.  On the other hand, since $\xi$
defines a non-empty wall
\[ 0 \leq c_2 + \frac{\xi^2-c_1^2}{4}=c_2-\frac{\sigma^2}{2}(2c_2-\alpha-2n) +\frac{e}{4}(1-\sigma^2)-\frac{\alpha}{2} \]
which gives us $\sigma= \pm 1$. Finally, $\xi L_n > 0$ implies
$\sigma = 1$ and thus $\xi=C_0-(2c_2-\alpha-2n)F=\xi_n$ which
proves Claim 1.

\vspace{3mm} Applying Claim 1, equation (\ref{des1}) turns out to
be
\begin{equation}
  \label{des2} \hspace{4mm}
 M_{L_n}(2;C_0+ \alpha F,c_2)=(M_{L_{n+1}}(2;C_0+ \alpha F,c_2) \setminus
 \cE_{\xi_n }(c_{1},c_2)) \sqcup \cE_{-\xi_n}(c_1,c_2).
   \end{equation}

Notice that
\[\begin{array}{ll}\tilde{c_1}L_n & =  (C_0+(\alpha+2n-2c_2)F)(C_0+(e+2c_2-\alpha-2n+1)F
) \\ & = 1\\ &  > 2(-4c_2+4n+2 \alpha -e-4)
\\ & = 2(-2C_0-(e+2)F)(C_0+(e+2c_2-\alpha-2n+1)F )\\ & = 2 K_X
L_n \end{array} \] and
\[\begin{array}{ll}\bar{c_1}L_{n+1} & =(-C_0+(2c_2-\alpha-2n)F)(C_0+(e+2c_2-\alpha-2n-1)F
)\\ & = 1 \\ &  > 2(-4c_2+4n+2 \alpha -e)
\\ & =
2(-2C_0-(e+2)F)(C_0+(e+2c_2-\alpha-2n-1)F ) \\ & =2 K_X L_{n+1}.
\end{array} \] Thus by Corollary \ref{superficie} the
Brill-Noether stratification of the moduli spaces
$M_{L_n}(2;\tilde{c_1},n)$ and $M_{L_{n+1}}(2;\bar{c_1},n)$ are
defined.

\noindent {\bf Claim 2:} We have: \[  (a) \quad \cE_{-\xi_n} \cong
W^1_{L_{n}}(2;\tilde{c_1},n), \]
\[(b) \quad \cE_{\xi_n} \cong  W^1_{L_{n+1}}(2;\bar{c_1},n). \]

\noindent {\bf Proof of Claim 2:} First of all notice that Claim 2
is, by definition, equivalent to have
\[  (a') \quad
\cE_{-\xi_n} \cong \{G \in M_{L_n}(2;C_0+ \alpha F,c_2)|
h^0(G(-(c_2-n)F))>0 \},\]
\[ (b') \quad \cE_{\xi_n} \cong \{G \in M_{L_{n+1}}(2;C_0+ \alpha F,c_2)|
h^0(G(-C_0+(c_2-n-\alpha )F))>0 \}. \] Let us prove $(a')$. If $E
\in \cE_{-\xi_n}$, then $E$ is given by a non-trivial extension
\[ 0 \rightarrow  \cO_X((c_2-n)F) \rightarrow   E  \rightarrow
 \cO_X(C_0-(c_2-n-\alpha)F)\otimes I_{Z} \rightarrow 0  \]
where $Z$ is a $0$-dimensional scheme of length
$|Z|=c_2(E(-(c_2-n)F))=n$. Therefore $h^0(E(-(c_2-n)F))>0$.
Moreover, it follows from (\ref{des2}) that $\cE_{-\xi_n} \subset
M_{L_n}(2;C_0+ \alpha F,c_2)$. Now let us see the converse. Let $E
\in \{G \in M_{L_n}(2;C_0+ \alpha F,c_2)| h^0(G(-(c_2-n)F))>0 \}$
and we are going to see that $E \in \cE_{-\xi_n}$. Let $s$ be a
non-zero section of $E(-(c_2-n)F)$ and let $Y$ be its scheme of
zeros. Let $D$ be the maximal effective divisor contained in $Y$.
Then $s$ can be regarded as a section of $E(-(c_2-n)F-D)$ and its
scheme of zeros has codimension $\geq 2$. Thus, for some effective
divisor $D=aC_0+bF$ we have a short exact sequence
\[ 0 \rightarrow  \cO_X((c_2-n)F+D) \rightarrow   E  \rightarrow
 \cO_X(C_0-(c_2-n-\alpha)F-D)\otimes I_{Z} \rightarrow 0.  \]
 By assumption, $E$ is $L_n$-stable. Therefore,
 \[ ((c_2-n)F+D)L_n < \frac{c_1(E)L_n}{2}=\frac{(C_0+ \alpha F)L_n}{2},\]
 which is equivalent to $a(2c_2-\alpha-2n+1)+b \leq 0$. Since $D$
 is an effective divisor $a,b \geq 0$ and hence the only solution
 is $a=b=0$. Therefore $D=0$ and $E$ is given by the exact
 sequence
 \[ 0 \rightarrow  \cO_X((c_2-n)F) \rightarrow   E  \rightarrow
 \cO_X(C_0-(c_2-n-\alpha)F)\otimes I_{Z} \rightarrow 0  \]
which proves that $E \in \cE_{-\xi_n}$.

Let us now prove $(b')$. By Remark  (\ref{des2}),  $\cE_{\xi_n}
\subset M_{L_{n+1}}(2;C_0+ \alpha F,c_2)$ and since any  $E \in
\cE_{\xi_n}$ is given by a non-trivial extension of type
\[ 0 \rightarrow  \cO_X(C_0-(c_2-n-\alpha)F) \rightarrow   E  \rightarrow
 \cO_X((c_2-n)F)\otimes I_{Z} \rightarrow 0,  \]
 for any $E \in \cE_{\xi_n}$, we have $h^0(E(-C_0+(c_2-n-\alpha )F))>0$.  Conversely,
given  $$E \in \{G \in M_{L_{n+1}}(2;C_0+ \alpha F,c_2)|
 h^0(G(-C_0+(c_2-n-\alpha )F))>0\}$$ we are going to see that $E \in \cE_{\xi_n}$. Let $s$ be
a non-zero section of $E(-C_0+(c_2-n-\alpha )F)$ and let $Y$ be
its scheme of zeros. Let $D$ be the maximal effective divisor
contained in $Y$. Then $s$ can be regarded as a section of
$E(-C_0+(c_2-n-\alpha )F-D)$ and its scheme of zeros has
codimension $\geq 2$. Thus, for some effective divisor $D=aC_0+bF$
we have a short exact sequence
\[ 0 \rightarrow  \cO_X(C_0-(c_2-n-\alpha)F+D) \rightarrow   E  \rightarrow
 \cO_X((c_2-n)F-D)\otimes I_{Z} \rightarrow 0.  \]
 By assumption, $E$ is $L_{n+1}$-stable. Therefore,
 \[ (C_0-(c_2-n-\alpha)F+D)L_{n+1} < \frac{c_1(E)L_{n+1}}{2}=\frac{(C_0+ \alpha F)L_{n+1}}{2},\]
 which is equivalent to $a(2c_2-\alpha-2n-1)+b \leq 0$. Since $D$
 is an effective divisor this implies that $a=b=0$ and thus $E$ is given by the exact
 sequence
 \[ 0 \rightarrow  \cO_X(C_0-(c_2-n-\alpha)F) \rightarrow   E  \rightarrow
 \cO_X((c_2-n)F)\otimes I_{Z} \rightarrow 0  \]
which proves that $E \in \cE_{\xi_n}$.

\vspace{3mm}

Putting together Claim 2 and (\ref{des2}) we obtain
\[ M_{L_n}(2;C_0+\alpha F,c_2) \cong (M_{L_{n+1}}(2;C_0+\alpha F,c_2) \setminus W^1_{L_{n+1}}(2;\bar{c_1},n))
\sqcup W^1_{L_{n}}(2;\tilde{c_1},n).\]

Finally let us see that the dimensions of the Brill-Noether loci
$W^1_{L_{n}}(2;\tilde{c_1},n)$ and $W^1_{L_{n+1}}(2;\bar{c_1},n)$
coincide with the expected one. To this end, by Claim 2, it
suffices to prove that \[ (i) \quad
\dim(\cE_{-\xi_n})=\rho_{L_{n}}^1(2;\tilde{c_1},n),
\]
\[(ii) \quad \dim(\cE_{\xi_n})= \rho_{L_{n+1}}^1(2;\bar{c_1},n). \]
By construction,
\[\begin{array}{ll}\dim \cE_{-\xi_n} & = ext^1(I_Z(C_0-(c_2-n-\alpha)F),\cO_X((c_2-n)F))+2|Z|-
h^0E(-(c_2-n)F) \\ & = h^1(I_Z(-C_0-(2c_2-2n-\alpha+e+2)F))+2|Z|-
h^0E(-(c_2-n)F)
\end{array} \]
being $E \in \cE_{-\xi_n}$. Since
$h^i(\cO_X(-C_0-(2c_2-2n-\alpha+e+2)F))=0$, for $i=0,1$, from the
cohomological exact sequence associated to
\[0 \rightarrow I_Z(-C_0-(2c_2-2n-\alpha+e+2)F) \rightarrow \cO_X(-C_0-(2c_2-2n-\alpha+e+2)F)\]\[ \rightarrow
\cO_Z(-C_0-(2c_2-2n-\alpha+e+2)F)\rightarrow 0 \] we deduce that
$$h^1(I_Z(-C_0-(2c_2-2n-\alpha+e+2)F))=h^0(\cO_Z(-C_0-(2c_2-2n-\alpha+e+2)F))=|Z|.$$
For any $E \in \cE_{-\xi_n}$, we have $h^0E(-(c_2-n)F)=1$ and
$c_2(E(-(c_2-n)F))=n=|Z|$. Putting altogether we get
\[ \dim \cE_{-\xi_n}=3n-1.\]
On the other hand, by Riemann-Roch theorem,
\[\begin{array}{ll} \chi(2;\tilde{c_1},n) &
=2 +\frac{(\tilde{c_1})(2C_0+(e+2)F)}{2}+\frac{(\tilde{c_1})^2}{2}-n \\
& =2
+\frac{(C_0+(\alpha+2n-2c_2)F)(2C_0+(e+2)F)}{2}+\frac{(C_0+(\alpha+2n-2c_2)F)^2}{2}-n
\\ &=3n-4c_2+2 \alpha -e +3, \end{array}\]
and by Proposition \ref{moduli} $\dim
M_{L_n}(2;\tilde{c_1},n)=4c_2-(2 \alpha-e)-3$.  Therefore,
\[\rho_{L_{n}}^1(2;\tilde{c_1},n)=\dim M_{L_n}(2;\tilde{c_1},n)-1(1-\chi(2,\tilde{c_1},n))=3n-1 \]
which proves $(i)$.

Let us now prove $(ii)$. By construction,
\[\begin{array}{ll}\dim \cE_{\xi_n} & = ext^1(I_Z((c_2-n)F),\cO_X(C_0-(c_2-n-\alpha)F)+2|Z|-
h^0E(-C_0+(c_2-n-\alpha)F) \\ &=
h^1(I_Z(-3C_0+(2c_2-2n-\alpha-e-2)F))+2|Z|-h^0E(-C_0+(c_2-n-\alpha)F)
\end{array} \]
being $E \in \cE_{\xi_n}$. Since
$h^i(\cO_X(-3C_0+(2c_2-2n-\alpha-e-2)F))=0$, for $i=0,2$,
$h^1(\cO_X(-3C_0+(2c_2-2n-\alpha-e-2)F))  =
-\chi(\cO_X(-3C_0+(2c_2-2n-\alpha-e-2)F))=-\chi$ and by
Riemann-Roch theorem
$$ \begin{array}{ll}  \chi
&=1+\frac{(-3C_0+(2c_2-2n-\alpha-e-2)F)(2C_0+(e+2)F)}{2}+\frac{(-3C_0+(2c_2-2n-\alpha-e-2)F)^2}{2}
\\ &= 4c_2-4n-2 \alpha-2+e. \end{array}$$
From the cohomological exact sequence associated to
\[0 \rightarrow I_Z(-3C_0+(2c_2-2n-\alpha-e-2)F) \rightarrow
\cO_X(-3C_0+(2c_2-2n-\alpha-e-2)F)\] \[ \rightarrow
\cO_Z(-3C_0+(2c_2-2n-\alpha-e-2)F)\rightarrow 0
\] we deduce that
$$\begin{array}{ll}h^1(I_Z(-3C_0+(2c_2-2n-\alpha-e-2)F)) &=
h^0(\cO_Z(-C_0-(2c_2-2n-\alpha+e+2)F))
\\ & +h^1(\cO_X(-3C_0+(2c_2-2n-\alpha-e-2)F))
\\ &=|Z|+4c_2-4n-2 \alpha-2+e. \end{array}$$ For any $E \in \cE_{\xi_n}$, we have $h^0E(-C_0+(c_2-n-\alpha)F)=1$ and
$c_2(E(-C_0+(c_2-n-\alpha)F)=n=|Z|$. Putting altogether we get
\[ \dim \cE_{\xi_n}=4c_2-n+e-2\alpha-3.\]
On the other hand, by Riemann-Roch theorem,
\[\chi(2;\bar{c_1},n)= 2+ \frac{(\bar{c_1})(2C_0+(e+2)F)}{2}+\frac{\bar{c_1}^2}{2}-n=1-n, \]
and by Proposition \ref{moduli} $\dim
M_{L_n}(2;\bar{c_1},n)=4c_2-(2 \alpha-e)-3$.  Thus,
\[\rho_{L_{n+1}}^1(2;\bar{c_1},n)=\dim M_{L_{n+1}}(2;\bar{c_1},n)-1(1-\chi(2;\bar{c_1},n))= 4c_2-n+e-2\alpha-3\]
which proves $(ii)$.
\end{proof}

\section{Brill-Noether loci and Instanton bundles}

In this section, we will address the three problems posed in
section two for the case of mathematical instanton bundles on
$\PP^3$. More precisely, we will prove that the Brill-Noether
locus of instanton bundles on $\PP^3$ is non-empty if and only if
the corresponding generalized Brill-Noether number is non-negative
and in this case the Brill-Noether locus is a smooth irreducible
variety of the expected dimension.

\vspace{3mm}

Let $MI(n)$ be the moduli space of mathematical instanton bundles
$E$ over $\PP^3$ with $c_1(E)=2$ and $c_2(E)=n$ and verifying the
instanton condition $H^1(E(-3))=0$. It is known that $MI(n)$ is
non-singular and irreducible for $n \leq 6$ (see \cite{B1} for
$n=2$, \cite{H2} for $n=3$, \cite{ES} for $n=4$, \cite{B2} and
\cite{LP} for $n=5$ and \cite{CTT} for $n=6$) and it has been
conjectured that $MI(n)$ is non-singular and irreducible. The only
known component $MI_0(n)$ of $MI(n)$ is the one made of vector
bundles which are generalizations of (a twist of) the ones
associated to $n+1$ skew lines in $\PP^3$ ('t Hooft bundles) and
we know that $MI_0(n)$ is a generically smooth variety of
dimension $8n-11$.

\vspace{3mm}

Since for any instanton bundle $E\in MI_0(n)$, we have
$H^{i}(\PP^3, E)=0$ for $i\ge 2$, Theorem \ref{construccio}
applies and the Brill-Noether stratification of $MI_0(n)$ is well
defined. So, for any $k\ge 1$, we can study the Brill-Noether loci
\[ W^k:= \{E \in MI_0(n) | h^0(E) \geq k \} \subset MI_0(n).\]
\noindent We have:

\begin{proposition}
\label{instanton} Assume $n>13$. With the above notations we have
\begin{itemize}
\item[(i)] $W^k \neq \emptyset$ if and only if $\rho^k(2;2,n) \geq
0$ if and only if $k < 3$. \item[(ii)] $W^1$ is a smooth rational
irreducible quasi-projective variety of the expected dimension
$5n-1$.  \item[(iii)] $W^2$ is a smooth rational irreducible
quasi-projective variety of the expected dimension $2n+7$.
\end{itemize}
\end{proposition}
\begin{proof} (i) Let us first compute $\rho^k(2;2,n)$. To this end, we take $E$ a
mathematical instanton bundle with Chern classes
$(c_1(E),c_2(E))=(2,n)$. It is well-known that $E(-1)$ is the
cohomology bundle of a monad of the following type
$$ 0 \longrightarrow {\cO}_{\PP^3}(-1)^{n-1} \mapright{\alpha
}{\cO}_{\PP^3}^{2n}  \mapright{\beta} {\cO}_{\PP^3}(1)^{n-1}
\longrightarrow 0.$$ Hence we have two exact sequences
$$ 0 \longrightarrow \cK=\ker (\beta)  \rightarrow {\cO}_{\PP^3}^{2n}  \mapright{\beta} {\cO}_{\PP^3}(1)^{n-1}
\longrightarrow 0,$$ and
$$ 0\longrightarrow {\cO}_{\PP^3}(-1)^{n-1} \mapright{\alpha
} \cK \rightarrow E(-1)  \longrightarrow 0$$ which allows us to
compute $\chi(E)$. Indeed,
\[\chi(E)= \chi(\cK(1))-(n-1)=2n\chi(\cO_{\PP^3}(1))-(n-1)\chi(\cO_{\PP^3}(2))-(n-1)=-3n+11. \]
Since $\dim MI_0(n)=8n-11$, we deduce that
\[\begin{array}{ll} \rho^k(2;2,n) & =\dim MI_0(n)-k(k- \chi(E)) \\ & =8n-11-k(k+3n-11) \\ &=n(8-3k)+k(11-k)-11.
\end{array} \] So, $\rho^1(2;2,n)=5n-1$, $ \rho^2(2;2,n)=2n+7$ and
$\rho^k(2;2,n)<0$ for $k \geq 3$ provided $n \geq 14$. By a result
in \cite{BG}, for any $E \in MI_0(n)$, $h^0(E) \leq 2$. On the
other hand, any instanton bundle $E$ associated to $n+1$ general
skew lines in $\PP^3$ satisfies $h^0(E)=1$ ('t Hooft bundles) and
any instanton bundle $E$ associated to $n+1$ skew lines on a
smooth quadric $Q$ in $\PP^3$ satisfies $h^0(E)=2$ (special 't
Hooft bundles). Putting altogether we have $W^k \neq \emptyset$ if
and only if $k \leq 2$, if and only if $\rho^k(2;2,n) >0$.

(ii) and (iii) The expected dimensions are $\rho^1(2;2,n)=5n-1$
and $ \rho^2(2;2,n)=2n+7$. Hence, the results follows from the
fact that $W^1$ (resp. $W^2$) is nothing but the variety of 't
Hooft bundles (resp. special ' t Hooft bundles) and we well-known
that the variety of 't Hooft bundles (resp. special ' t Hooft
bundles) is a smooth, irreducible and rational variety of
dimension $5n-1$ (resp. $2n+7$).

\end{proof}

\end{document}